\documentclass[11pt,a4paper]{amsart}

\usepackage{amsfonts,amsthm,amsmath,amssymb,amscd,mathrsfs}
\usepackage{graphics}
\usepackage{indentfirst}
\usepackage{cite}
\usepackage{bm, enumerate,xcolor}
\usepackage{enumitem}
\usepackage[dvips]{epsfig}
\usepackage{latexsym}
\usepackage{float}
\usepackage{mathtools}
\usepackage[hidelinks]{hyperref} 
\usepackage{appendix}
\usepackage{soul}

\usepackage{accents}

\usepackage [english]{babel}
\usepackage [autostyle, english = american]{csquotes}
\MakeOuterQuote{"}

\calclayout

\newcommand{\wto}{\rightharpoonup}

\newcommand{\Rn}[1]{\expandafter{\romannumeral #1\relax}}
\newcommand{\RN}[1]{\uppercase\expandafter{\romannumeral #1\relax}}

\newtheorem{theorem}{Theorem}[section]
\newtheorem{remark}{Remark}[section]

\newtheorem{lemma}{Lemma}[section]
\newtheorem{pro}[lemma]{Proposition}
\newtheorem{coro}[lemma]{Corollary}

\newtheorem{assumption}{Assumption}[section]

\def\N{\mathbb{N}}
\def\R{\mathbb{R}}
\def\mbN{\mathbb{N}}
\def\mbR{\mathbb{R}}

\def\e{\varepsilon}

\def\bv{\mathbf{v}}

\def\bn{\mathbf{n}}

\def\mcH{\mathcal{H}}

\def\mcX{\mathcal{X}}

\def\Div{{\rm div}}

\renewcommand{\S}{\mathbb{S}}

\def\XXint#1#2#3{{\setbox0=\hbox{$#1{#2#3}{\int}$ }
		\vcenter{\hbox{$#2#3$ }}\kern-.6\wd0}}


\makeatletter
\@namedef{subjclassname@2020}{\textup{2020} Mathematics Subject Classification}
\makeatother

\begin{document}

\title[Steady Euler flow with least total curvature]{Least total curvature solutions to steady Euler system and monotone solutions to semilinear equations in a strip}

\author{Changfeng Gui}
\address{Department of Mathematics, Faculty of Science and Technology, University of Macau, Taipa, Macao.}
\email{changfenggui@um.edu.mo}

\author{David Ruiz}
\address{IMAG, Departamento de An\'alisis Matem\'atico, Universidad de Granada, 18071 Granada, Spain}
\email{daruiz@ugr.es}

\author{Chunjing Xie}
\address{School of Mathematical Sciences, Institute of Natural Sciences, Ministry of Education Key Laboratory of Scientific and Engineering Computing, CMA-Shanghai, Shanghai Jiao Tong University, Shanghai 200240, China.}
\email{cjxie@sjtu.edu.cn}

\author{Huan Xu}
\address{Department of Mathematics, Faculty of Science and Technology, University of Macau, Taipa, Macao.}
\email{hzx0016@auburn.edu}


\date{}

\maketitle

\begin{abstract}
This paper focuses on establishing the existence of a class of steady solutions, termed {\it least total curvature solutions}, to the incompressible Euler system in a strip. The solutions obtained in this paper complement the 
{\it least total curvature solutions} already known. Our approach employs a minimization procedure to identify a monotone heteroclinic solution for a conveniently chosen semilinear elliptic PDE. This method also enables us to construct positive and monotone (and consequently stable) solutions to semilinear elliptic PDEs with non-convex superlevel sets in a strip domain. This can be regarded as a negative answer to a generalized problem raised in \cite{HNS16}. 
\end{abstract}



\section{Introduction}

In this paper, our main concern is steady solutions to the incompressible Euler system in a domain $\Omega\subset\mbR^2$
\begin{eqnarray}\label{Euler}
\left\{\begin{aligned}
&\bv\cdot\nabla\bv+\nabla P=0, & {\rm in}\ \Omega,\\
&\Div \, \bv=0, & {\rm in}\ \Omega.
\end{aligned}\right.
\end{eqnarray}
Here $\bv=(v_1,v_2)$ is the velocity field and $P$ is the pressure. The flow is supplemented with the {\it slip boundary condition}
\begin{align}\label{slip boundary condition}
\bv\cdot\bn=0\ \ {\rm on}\ \partial\Omega,
\end{align}
where $\bn$ denotes the unit outer normal to $\partial\Omega$.

It is well-known that both parallel shear flows (or simply shear flows) and circular flows are solutions to \eqref{Euler}. Here, a shear flow has the form $\bv=(v_1(x_2),0)$ up to rotation, while a circular flow takes the form $\bv=V(|x|)\frac{(-x_2,x_1)}{|x|}$ for some function $V$. 

 In recent years, there has been a growing interest in the rigidity issue for steady Euler flows in domains with symmetry. The question within this framework asks whether solutions inherit the symmetries of the domain. 
Regarding radial symmetry, it was proved in \cite{hamel2021circular} that sufficiently regular steady Euler solutions with non-vanishing velocity on an annulus must be circular.
Another notable recent result in \cite{gomez2021symmetry} establishes the radial symmetry of steady solutions characterized by compactly supported, non-negative vorticity. Radial symmetry of compactly supported velocities has been established by \cite{RuizARMA} under some assumptions on the stagnation set: these assumptions are necessary by the flexibility result of \cite{EFR}. A recent result in \cite{EHSX} demonstrates that analytic flows in simply connected analytic domains must be circular if their stream functions cannot be described by a global semilinear elliptic equation.

 In this paper, we are interested in steady flows in domains that are invariant in one direction, where the symmetric solutions are shear flows. 
In \cite{hamel2017shear,hamel2019liouville}, Hamel and Nadirashvili establish a number of rigidity results on the strip, the half-plane, and the entire plane by assuming essentially that the velocity field does not vanish on the domain. 
The rigidity of certain special shear flows is investigated in \cite{coti2023stationary,lin2011inviscid,GXX_arxiv_2024}, addressing the existence and non-existence of non-shear structures near these given flows.
See also the work in \cite{Li2022Poiseuille} on the rigidity of Poiseuille flows in a strip.

Further results on flexibility and rigidity in various contexts can be found in \cite{Cao, CastroLear1,CastroLear2, constantin2021flexibility, drivas2023islands, drivas-nualart, nualart2023zonal}.

Adopting a novel global perspective, the work \cite{GXX_arxiv_2024} provides a complete classification of bounded steady Euler flows in the entire plane \(\mathbb{R}^2\), the half-plane \(\mathbb{R}^2_+ := \mathbb{R} \times (0, \infty)\), and the infinitely long strip \(\Omega_\infty := \mathbb{R} \times (0, 1)\)—based on the flow angle structure. For flows in the half-plane or strip, the key findings of \cite{GXX_arxiv_2024} are summarized below.

\begin{theorem}[\hspace{1sp}\cite{GXX_arxiv_2024}]\label{GXX_classification}
Let $\Omega=\Omega_\infty$ or $\mbR_+^2$. Suppose that $\bv\in C^2(\Bar{\Omega})$ is a bounded solution of \eqref{Euler} satisfying \eqref{slip boundary condition}. Then exactly one of the following situations holds.
\begin{enumerate}[label=(\Roman*)]
\item $\bv$ is a parallel shear flow, that is, $\bv\equiv (v_1(x_2),0)$;

\item $\Theta(\bv;\Bar{\Omega})\vcentcolon=
\left\{\frac{\bv(x)}{|\bv(x)|}:\  x\in \Bar{\Omega}, \ |\bv(x)|>0\right\}=\S^1$ (unit circle);

\item $\Theta(\bv;\Bar{\Omega})=\S_+^1=\{(\cos \theta, \sin \theta):\theta\in [0, \pi]\}$ (upper closed semicircle), or $\Theta(\bv;\Bar{\Omega})=\S_-^1=\{(\cos \theta, \sin \theta):\theta\in [ \pi, 2\pi]\}$ (lower closed semicircle).
\end{enumerate}
\end{theorem}

Examples of flows of type (I) and (II) are illustrated in Figures \ref{shear}, \ref{circular}, respectively. Flows of type (III) will be discussed below.

\begin{figure}[h]
\centering
\includegraphics[height=5cm, width=10cm]{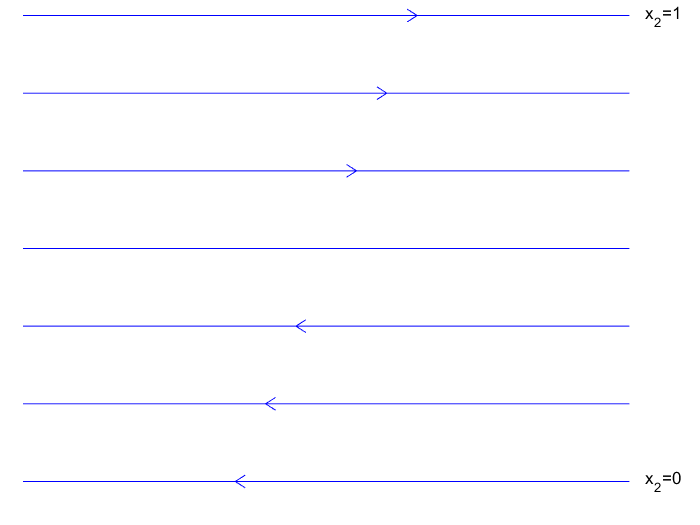}
\caption{Flows of type \RN{1}, $\bv=(x_2-\frac{1}{2}, 0)$.}
\label{shear}
\end{figure}

\begin{figure}[h]
\centering
\includegraphics[height=5cm, width=10cm]{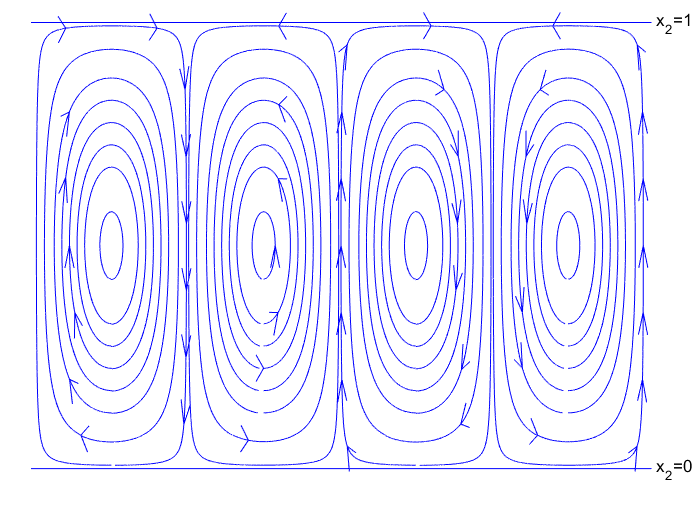}
\caption{Flows of type \RN{2}, $\bv=(-\pi \sin x_1\cos (\pi x_2), \cos x_1\sin (\pi x_2))$.}
\label{circular}
\end{figure}

The key of the proof for Theorem \ref{GXX_classification} lies in the estimate of the {\it total curvature}
\begin{align*}
\int_{\Omega}\frac{|v_1\nabla v_2-v_2\nabla v_1|^2}{|\bv|^2}\, dx,
\end{align*}
where the integrand is understood as $0$ in the {\it stagnation set} $\{\bv=0\}$. 
The estimate for the total curvature is based on the delicate energy estimate for the equation
\begin{align*}
\bv\cdot\nabla\omega=\Div(v_1\nabla v_2-v_2\nabla v_1)=0.
\end{align*}
Clearly, a flow is a parallel shear flow if and only if its total curvature is zero. Furthermore, a type (\RN{3}) flow is characterized as the non-shear extremal functions of an inequality concerning a sharp lower bound of the total curvature (see \cite{GXX_arxiv_2024}). Henceforth, we call type (\RN{3}) flow the {\it least total curvature flow}. It is noteworthy that each streamline in the {\it least total curvature flow} can turn at most 180 degrees, as opposed to 360 degrees. This characteristic, in essence, provides a local perspective that justifies the naming of this type of flow.

Besides the existence of the {least total curvature flow}, the following important properties for the {least total curvature flow} were also established in  \cite{GXX_arxiv_2024}. 

\begin{pro}[\hspace{1sp}\cite{GXX_arxiv_2024}]\label{search_guide}\label{pro_motiv}
Let $\bv=(v_1,v_2)\in C^1(\Bar{\Omega})$ be a bounded solution of \eqref{Euler} satisfying \eqref{slip boundary condition}. Assume $\Theta(\bv;\Bar{\Omega})=\S^1_+$.
\begin{enumerate}[label=(\roman*)]
\item  If $\Omega=\mbR_+^2$, then $v_1(\cdot,0)$ is nonincreasing in $\mbR$, and it holds that $\underline{v}_{1,-}=-\underline{v}_{1,+}>0$, and that 
\begin{align*}
\int_{\Omega}\frac{|v_1\nabla v_2-v_2\nabla v_1|^2}{|\bv|^2}\,dx=\frac{\pi}{2}\underline{v}^2_{1,+},
\end{align*}
where $\underline{v}_{1,\pm}=\lim_{x_1\to \pm\infty} v_1(x_1,0).$

\item  If $\Omega=\Omega_\infty$, then $v_1(\cdot,1)$ is nondecreasing in $\mbR$ and $v_1(\cdot,0)$ is nonincreasing in $\mbR$, and it holds that
\begin{align}\label{balancing_law}
\bar{v}_{1,+}^2-\bar{v}_{1,-}^2=\underline{v}_{1,+}^2-\underline{v}_{1,-}^2,
\end{align}
and 
\begin{align}\label{total_curvature_strip}
\int_{\Omega}\frac{|v_1\nabla v_2-v_2\nabla v_1|^2}{|\bv|^2}\,dx
=\frac{\pi}{4}\left(\bar{v}_{1,+}|\bar{v}_{1,+}|-\bar{v}_{1,-}|\bar{v}_{1,-}|+\underline{v}_{1,-}|\underline{v}_{1,-}|-\underline{v}_{1,+}|\underline{v}_{1,+}|\right),
\end{align}
where $\underline{v}_{1,\pm}=\lim_{x_1\to \pm \infty} v_1(x_1,0)$ and $\bar{v}_{1,\pm}=\lim_{x_1\to \pm \infty} v_1(x_1, 1).$
\end{enumerate}
\end{pro}

Motivated by Proposition \ref{pro_motiv}, the main scope of this paper is to study different types of least total curvature solutions based on the signs of $v_1$ on the physical boundary $\partial\Omega$.
Observe that if $\Omega = \mbR_+^2$, the function $v_1(\cdot, 0)$ changes sign exactly once, and this represents a unique type of least total curvature flow.
So we shall only consider the flows in strips (i.e., $\Omega=\Omega_\infty$) in the rest of this paper.

Assume $\bv\in C^1(\overline{\Omega_\infty})$ is a least total curvature flow with $\Theta(\bv;\overline{\Omega_\infty})=\S^1_+$. In view of \eqref{balancing_law} and \eqref{total_curvature_strip}, if all elements of $\mathcal{V}\vcentcolon=\{\bar{v}_{1,+}, \bar{v}_{1,-}, \underline{v}_{1,+}, \underline{v}_{1,-}\}$ are nonnegative or nonpositive, we infer that the total curvature must be zero so that $\bv$ is a shear flow. Hence, $\mathcal{V}$ has both strictly positive and strictly negative elements. Therefore, Proposition \ref{search_guide} yields the following corollary.

\begin{coro}
Assume $\bv\in C^1(\overline{\Omega_\infty})$ is a {\it least total curvature flow} with $\Theta(\bv;\overline{\Omega_\infty})=\S^1_+$. Then exactly one of the following cases holds.

\begin{itemize}
\item case (a): Both $v_1(\cdot,1)$ and $v_1(\cdot,0)$ change sign.
Since $v_1(\cdot,1)$ is nondecreasing and $v_1(\cdot,0)$ is nonincreasing, we infer that $\bar{v}_{1,-}<0<\bar{v}_{1,+}$ and $\underline{v}_{1,+}<0<\underline{v}_{1,-}$. 

\item case (b): Neither $v_1(\cdot,1)$ nor $v_1(\cdot,0)$ changes sign. Since $\tilde{\bv}(x)=(-v_1(-x_1,x_2),v_2(-x_1,x_2))$ is also a least total curvature flow with $\Theta(\tilde{\bv};\overline{\Omega_\infty})=\S^1_+$, we may essentially assume that $v_1(\cdot,1)\ge 0$ and $v_1(\cdot,0)\le 0$. Then it must hold that $\bar{v}_{1,+}>0$ and $\underline{v}_{1,+}<0$.

\item case (c): Only one of $v_1(\cdot,1)$ and $v_1(\cdot,0)$ changes sign. 
Since both $\tilde{\bv}(x)$ and  $\hat{\bv}(x)=\bv(-x_1,1-x_2)$ are least total curvature flows, we may assume that $v_1(\cdot,1)$ changes sign but $v_1(\cdot,0)\le 0$. Then the monotonicity of $v_1(\cdot,1)$ yields $\bar{v}_{1,-}<0<\bar{v}_{1,+}$.
\end{itemize}
\end{coro}

The existence of the least total curvature flows in case (a) has been established in \cite{GXX_arxiv_2024}, with the corresponding streamline pattern depicted in Figure \ref{case1}.
Subsequently, the existence for case (b) is established in \cite{DR_arxiv_2024}. See Figure \ref{case2} for the corresponding streamline pattern. The first main contribution of this paper is the construction of a least total curvature flow in case (c).

\begin{figure}[h]
	\centering
	\includegraphics[height=5cm, width=10cm]{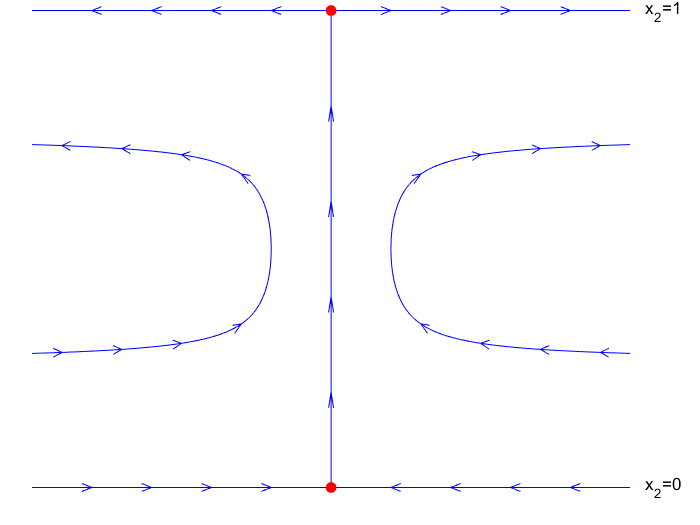}
	\caption{The least total curvature flows in case (a).}
	\label{case1}
\end{figure}

\begin{figure}[h]
\centering
\includegraphics[height=5cm, width=10cm]{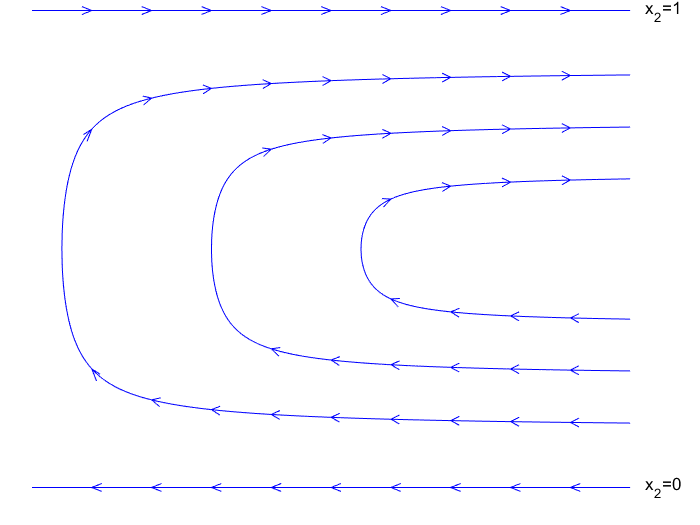}
\caption{The least total curvature flows in case (b).}
\label{case2}
\end{figure}

Our first theorem is precisely stated as follows.

\begin{theorem}\label{main theorem}
There exists a bounded solution $\bv\in C^\infty(\overline{\Omega_{\infty}})$ to \eqref{Euler}-\eqref{slip boundary condition} satisfying the following properties:
\begin{enumerate}
    \item[(i)] $v_2>0$ in $\Omega_{\infty}$;

\item[(ii)] $v_1(\cdot,0)\le \underline{v}_{1,-}<0$ on $\mbR$, $v_1(\cdot,1)>0$ on $(0,\infty)$, and $v_1(\cdot,1)<0$ on $(-\infty,0)$.
\end{enumerate}
\end{theorem}
See Figure \ref{case3} for the streamline pattern.

\begin{figure}[h]
\centering
\includegraphics[height=5cm, width=10cm]{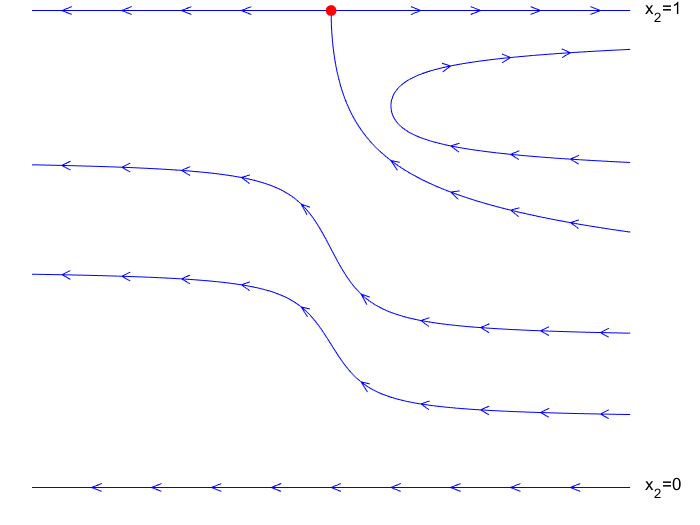}
\caption{The least total curvature flows in case (c).}
\label{case3}
\end{figure}

To prove Theorem \ref{main theorem}, we shall establish the existence of a monotone (in the $x_1$-direction) solution to the following boundary value problem of  semilinear elliptic equation
\begin{eqnarray}\label{slee}
\left\{\begin{aligned}
&-\Delta u=f(u), & x\in \Omega_\infty,\\
&u(x_1,0)=0, & x_1\in \mbR,\\
&u(x_1,1)=c, & x_1\in \mbR,
\end{aligned}\right.
\end{eqnarray}
where $c$ is a nonnegative constant and $f$ is a particularly chosen smooth function. Define
\[
F(u)=\int_0^u f(s)\,ds.
\]
Then $(\bv,P)$ defined by
\begin{align}\label{euler_solution}
(v_1,v_2,P)=\left(-\partial_{x_2}u,\partial_{x_1}u, -F(u)-\frac12|\nabla u|^2\right)    
\end{align}
solves \eqref{Euler}-\eqref{slip boundary condition}. Hence the key issue is to prove the existence of solutions to \eqref{slee}, which satisfy $\partial_{x_1}u>0$ in $\Omega_\infty$ and 
\[
\partial_{x_2}u(x_1,0)\ge -\underline{v}_{1,-}>0\,\,\text{for}\,\,x_1\in \mbR, \,\, \partial_{x_2}u(x_1,1)<0\,\,\text{for}\,\, x_1>0,\,\,\text{and}\,\,  \partial_{x_2}u(x_1,1)>0\,\,
\text{for}\,\,x_1<0.
\]

Furthermore, inspired by our analysis for Theorem \ref{main theorem}, we can prove the following result.

\begin{theorem} \label{theostable} 
There exists a $C^{\infty}$ function $f$ and a positive solution $u$ of \eqref{slee} with $c=0$. Furthermore, $u$ is monotone along the $x_1$ direction and has at least one non-convex superlevel set $\{x \in \Omega_{\infty}: \ u(x) > \alpha\}$ for some $\alpha >0$.
\end{theorem}

The streamline pattern for the corresponding Euler flow is illustrated by Figure \ref{non_convex}. Note that this flow also belongs to case (b).

\begin{figure}[h]
\centering
\includegraphics[height=5cm, width=10cm]{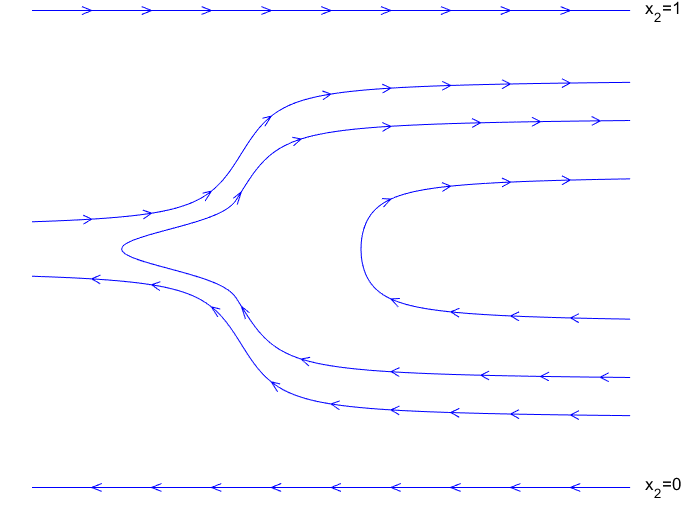}
\caption{Nonconvex streamlines.}
\label{non_convex}
\end{figure}

We have the following remarks on Theorem \ref{theostable}.

\begin{remark}
In \cite[Remark 3 in page 268]{Lions}, P. L. Lions writes that, in a bounded convex domain $\Omega$, ``We believe that ... for general $f$, the (super)level sets of
any (positive solution of semilinear elliptic equation with zero Dirichlet boundary condition)  are convex''.  This problem was restated in  \cite{Brezis99,CC98}).
There is indeed a vast literature containing
some proofs of this longstanding problem for various nonlinearities $f$. 
 In particular, the superlevel sets are convex if $f=1$ (\hspace{1sp}\cite{ML71}), $f= \lambda_1 u$ (\hspace{1sp}\cite{BL76}) (with $\lambda_1$ the principal eigenvalue of the Dirichlet Laplacian), or $f(u)= u^p$ with $p \in (0,1)$ (\hspace{1sp}\cite{K85}). Quite surprisingly, the convexity of superlevel sets can fail for some function $f(u)$ with particually chosen domain, and a counter-example has been found by Hamel, Nadirashvili, and Sire in \cite{HNS16}. This example, though, is found via a constrained minimization, and one would expect it to be unstable. Here by stability, we mean that the quadratic form associated with the linearized operator is nonnegative. As a consequence, it is rather natural to conjecture the convexity of superlevel sets for positive stable solutions to semilinear elliptic equation in bounded convex domains with zero Dirichlet boundary condition. This is also raised by Hamel, Nadirashvili, and Sire as an open problem. More precisely, in \cite[Remark 1.3 in page 502]{HNS16}, they mention that ``proving the quasiconcavity from the semi-stability is
still an open question''.  The solution constructed in Theorem \ref{theostable} is positive and monotone (which is known to imply stability), and hence serves as a counterexample to the above mentioned question in \cite{HNS16} in the \emph{unbounded} convex domain $\Omega_{\infty}$. 
\end{remark}

\begin{remark}
    It should be noted that  geometric properties of streamlines, curvature and convexity, etc.,  play an important role in understanding both incompressible and compressible flows; see \cite[Conjecture 1]{CJS}, \cite{TianCui}, etc.
\end{remark}

Here we give the key ideas for the proof of Theorems \ref{main theorem} and  \ref{theostable}. The proof takes two steps. First, we look for solutions to \eqref{slee} as heteroclinic solutions connecting certain functions $\underline{u}$ and $\bar{u}$. As in \cite{DR_arxiv_2024}, this can be done as long as $\underline{u}$ and $\bar{u}$ are energy minimizers without other minimizers in between. In the second step, our aim is to find a suitable nonlinearity $f$ so that the one-dimensional energy minimizers $\bar{u}$ and $\underline{u}$ have the desired properties. In \cite{DR_arxiv_2024}, $\underline{u}$ and $\bar{u}$ were chosen as 0 and a positive minimizer, respectively. To establish Theorem \ref{main theorem}, we construct a monotone minimizer $\underline{u}$ and a non-monotone minimizer $\bar{u}$ that lies above $\underline{u}$. Lastly, we modify $f$ so that \eqref{slee} has two positive one-dimensional minimizes, thus concluding the proof of Theorem \ref{theostable}.

The rest of the paper is organized as follows. In Section \ref{sec_heter}, following \cite{DR_arxiv_2024}, we construct monotone heteroclinic solutions connecting two energy minimizers when there does not exist another minimizer between them. Section \ref{sec_pf1} and Section \ref{sec_pf2} are devoted to the proof of Theorem \ref{main theorem} and Theorem \ref{theostable}, respectively.


\section{Schematic construction of monotone heteroclinic solutions}\label{sec_heter}

This section is devoted to the proof of the existence of monotone solutions of \eqref{slee} which connects two one-dimensional solutions under certain assumptions. 
Note that Theorem \ref{main theorem} is proved as long as one can establish the existence of a solution to \eqref{slee} such that
\[
\partial_{x_1}u \ge 0, \ \ \partial_{x_2}u(\cdot, 0) >0, \ \partial_{x_2}u(\cdot, 1) \mbox{ changes sign.}
\]
Because of the monotonicity along the $x_1$-direction, we can define
\[
\underline{u}(x_2)=\lim_{x_1\rightarrow-\infty}u(x_1,x_2),\quad
\bar{u}(x_2)=\lim_{x_1\rightarrow+\infty}u(x_1,x_2).
\]
At least formally, $\underline{u}< \bar{u}$ are two solutions to 
\begin{eqnarray}\label{ode}
	\left\{\begin{aligned}
		&-\psi''=f(\psi), & {\rm in}\ (0,1),\\
		&\psi(0)=0, \psi(1)=c.
	\end{aligned}\right.
\end{eqnarray}

 Roughly speaking, this can be done as long as $\underline{u}$ and $\overline{u}$ are two global minima of the energy functional related to \eqref{ode},
\begin{equation} \label{energyfunctional}
I(\psi)=\int_0^1\left(\frac12|\psi'(x_2)|^2- F(\psi(x_2))\right) \, dx_2,   
\end{equation}
 defined for all $\psi\in \mcH_{0,c}(0,1)$, where
\[
\mcH_{0,c}(0,1)=\{\psi\in H^{1}(0,1): \psi(0)=0, \psi(1)=c  \}.
\]

Being more specific, we will work under the following conditions.

\begin{assumption}\label{two_ends}
There are two global minimizers $\phi,\varphi\in \mcH_{0,c}(0,1)$  of $I$, defined in \eqref{energyfunctional},  such that the following two properties hold.
	
	(1) $\phi<\varphi$ in $(0,1)$. 
	
	(2) There does not exist any global minimizer between $\phi$ and $\varphi$, that is, if $\eta$ is a global minimizer satisfying $\phi\le \eta\le \varphi$, then either $\eta\equiv\phi$ or $\eta\equiv\varphi$. 
\end{assumption}

Under Assumption \ref{two_ends}, our goal is to construct a monotone solution $u$ to the problem \eqref{slee} that connects $\phi$ and $\varphi$. 
Specifically,  this section is devoted to the proof of the following proposition. 
\begin{pro}\label{pro_scheme}
	Suppose $f\in C^1$ is such that Assumption \ref{two_ends} is fulfilled. Then \eqref{slee} has a solution $u\in C^{2,\alpha}(\overline{\Omega_\infty})$ (for any $\alpha \in (0,1)$) that satisfies 
	\begin{align*}
		\partial_{x_1}u>0\quad in\ \Omega_\infty,
	\end{align*}
	and 
	\begin{align*}
		\lim_{x_1\rightarrow-\infty}u(x_1,x_2)=\phi(x_2),\quad
		\lim_{x_1\rightarrow+\infty}u(x_1,x_2)=\varphi(x_2)
	\end{align*}
	uniformly in $x_2\in[0,1]$.
\end{pro}

There are three remarks in order. 
\begin{remark}
The existence of heteroclinic solutions for elliptic PDEs is a question that has been intensively investigated in \cite{Ra94,Ra02} in a periodic setting: that is, the limit functions $\phi$ and $\varphi$ are periodic in the $x_N$ variable. Ours is a particular case of this study, since our limit functions are independent of $x_N$; however, we also obtain monotonicity of the solution in the $x_1$ variable, which is important in our application to the stationary solution of the 2D Euler equations. 
For further results in cylinders we refer also to~\cite{A19, BMR16, Ki82,Ra04}.
\end{remark}

\begin{remark}
We note that when the domain $\Omega=\mbR^2$ or $\mbR^2_{+}$,  similar problems to \eqref{slee}  with boundary conditions  replaced by possible uniform limit conditions at infinity in $x_2$ 
 have been studied intensively before.  In contrast to Proposition \ref{pro_scheme},  the solutions in $\Omega=\mbR^2$ or $ \mbR^2_{+}$  often enjoy one-dimensional symmetry  (see, e.g.  \cite{{BNV_CPAM_1994}, {BCN_ASNSP_1997}, {bbg}}), mainly due to the failure of {\bf Assumption \ref{two_ends}}  when the interval $(0,1)$ is replaced by unbounded ones. However,  regarding  some systems of PDEs  for which  there are two distinct global minimizers of the associated energy functional defined on $\mbR,$
there may exist a heteroclinic solution in $\mbR^2$  connecting the two global minimizers, very much like the solution in Proposition \ref{pro_scheme} (see \cite{abg}).  
 
\end{remark}

\begin{remark}
In fact, in the case $\phi=0$, Proposition \ref{pro_scheme} is exactly \cite[Theorem 1.1]{DR_arxiv_2024}. The proof of the general case follows the same lines as that of \cite[Theorem 1.1]{DR_arxiv_2024}.
\end{remark}

\begin{proof}[Proof of Proposition \ref{pro_scheme}] The proof is divided in several steps.

\medskip

{\bf Step 1. Approximating solutions.} Let $\Omega_n=(-n,n)\times(0,1)$ and define
\[
\mcX(\Omega_n)=\{u\in H^1(\Omega_n): u(\cdot,0)=0, u(\cdot,1)=c, u(-n,\cdot)=\phi, u(n,\cdot)=\varphi \}.
\]
Clearly, $\mcX(\Omega_n)$ is nonempty since it contains $\frac{x_1}{2n}(\varphi(x_2)-\phi(x_2))+\frac{1}{2}(\varphi(x_2)+\phi(x_2))$.
Define
\[
J_n(u)=\int_{\Omega_n}\left(\frac12|\nabla u|^2-F(u)\right)\,dx,\quad u\in \mcX(\Omega_n).
\]
Consider the minimization problem
\[
e_n\vcentcolon=\inf_{u\in \mcX(\Omega_n)}J_n(u).
\]
Obviously, $e_n<\infty$. It is also not difficult to see that $e_n>-\infty$ since it holds for every $u\in \mcX(\Omega_n)$ that
\begin{align*}
J_n(u)\ge& \int_{-n}^{n}\int_0^1\left(\frac12|\partial_{x_2} u|^2-F(u)\right) \,dx_2 \,dx_1\\
=&\int_{-n}^{n}I(u(x_1,\cdot)) \,dx_1\ge 2n I(\phi).
\end{align*}

Let $\{\tilde{u}_{n,k} \}_{k\in\mbN}\subset \mathcal{X}(\Omega_n)$ be a minimizing sequence of $J_n$. Define
\begin{eqnarray*}
u_{n,k}=\left\{\begin{aligned}
&\phi, & if\ \tilde{u}_{n,k}\le\phi,\\
&\tilde{u}_{n,k}, & if\ \phi<\tilde{u}_{n,k}<\varphi,\\
&\varphi, & if\ \tilde{u}_{n,k}\ge\varphi.
\end{aligned}\right.
\end{eqnarray*}
We claim that $\{u_{n,k} \}_{k\in\mbN}\subset \mathcal{X}(\Omega_n)$ is also a minimizing sequence of $J_n$. For this, we first observe that
\begin{align*}
J_n(\phi)-J_n(\phi\wedge \tilde{u}_{n,k})
\le
\int_{-n}^{n} \left( I(\phi)- I(\phi\wedge \tilde{u}_{n,k}(x_1,\cdot)) \right) \,dx_1
\le 0,
\end{align*}
where we used the notation $a\wedge b=\min\{a,b\}$.
Similarly, denoting $a\vee b=\max\{a,b\}$, we have
\begin{align*}
J_n(\varphi)-J_n(\varphi\vee \tilde{u}_{n,k})\le 0.
\end{align*}
Consequently,
\begin{align*}
J_n(u_{n,k})=&\int_{\{v_{n,k}\le\phi \}}\left(\frac12|\phi'(x_2)|^2-F(\phi(x_2))\right)\,dx
+\int_{\{\phi<\tilde{u}_{n,k}<\varphi \}}\left(\frac12|\nabla \tilde{u}_{n,k}|^2-F(\tilde{u}_{n,k})\right) \,dx\\
&+\int_{\{\tilde{u}_{n,k}\ge\varphi \}}\left(\frac12|\varphi'(x_2)|^2-F(\varphi(x_2))\right)\,dx\\
=& J_n(\tilde{u}_{n,k})+J_n(\phi)-J_n(\phi\wedge \tilde{u}_{n,k})+J_n(\varphi)-J_n(\varphi\vee \tilde{u}_{n,k})\\
\le & J_n(\tilde{u}_{n,k}).
\end{align*}
This finishes the proof of the claim.

Since $\phi\leq u_{n,k}\leq \varphi$,  $\{u_{n,k} \}_{k\in\mbN}$ is a uniformly bounded minimizing sequence of $J_n$. We infer that $\|u_{n,k}\|_{H^1(\Omega_n)}$ is uniformly bounded in $k$. So there exists some $u_n\in \mathcal{X}(\Omega_n)$ such that, up to a subsequence, $\{u_{n,k}\}$ converges weakly to $u_n$ in $H^1$, and $\{u_{n,k}\}$ converges pointwise to $u_n$. By the weak lower semicontinuity of the $H^1$-norm and the dominated convergence theorem, we conclude that $u_n$ is a minimizer such that $e_n=J(u_n)$.
In particular, $u_n$ is a solution to
\begin{eqnarray}\label{approx_equ}
\left\{\begin{aligned}
&-\Delta u_n=f(u_n), & x\in \Omega_n,\\
&u_n(x_1,0)=0, & -n\le x_1\le n,\\
&u_n(x_1,1)=c, & -n\le x_1\le n,\\
&u_n(-n,x_2)=\phi(x_2), & 0\le x_2\le 1,\\
&u_n(n,x_2)=\varphi(x_2), & 0\le x_2\le 1.
\end{aligned}\right.
\end{eqnarray}

Observe that  $\phi\le u_n\le \varphi$, by pointwise convergence. By the elliptic regularity estimates and a standard reflection argument (see, e.g., \cite[Proposition 2.2]{DR_arxiv_2024}), we have that
\begin{align}\label{uniform_bound}
\sup_{n\ge1}\|u_n\|_{C^{2,\alpha}(\overline{\Omega_{n}})}<\infty.
\end{align}

{\bf Step 2. Monotonicity.} We show that $\partial_{x_1}u_n>0$ in $\Omega_n$. Note that $\partial_{x_1}u_n$ is a weak solution to
\begin{eqnarray*}
\left\{\begin{aligned}
&-\Delta(\partial_{x_1}u_n)=f'(u_n)(\partial_{x_1}u_n), & {\rm in}\ \Omega_n,\\
&\partial_{x_1}u_n\ge 0, & {\rm on}\ \partial\Omega_n,
\end{aligned}\right.
\end{eqnarray*}
where the boundary condition follows from the boundary condition of $u_n$ and the fact that $\phi\le u_n\le \varphi$ in $\Omega_n$.

Since $u_n$ is a global minimizer of $J_n$, the principal eigenvalue $\lambda_1$ of the operator $-\Delta-f'(u_n)$, supplemented with homogeneous Dirichlet boundary condition, is nonnegative. If $\lambda_1>0$, the maximum principle holds for $-\Delta-f'(u_n)$ (see \cite{BNV_CPAM_1994}), that is, either $\partial_{x_1}u_n>0$ in $\Omega_n$ or $\partial_{x_1}u_n\equiv0$. But the latter cannot occur since $u_n(-n,\cdot)=\phi<\varphi=u_n(n,\cdot)$ in $(0,1)$. If $\lambda_1=0$, by \cite[Corollary 2.2]{BNV_CPAM_1994}, $\partial_{x_1}u_n$ is a constant multiple of the eigenfunction $\phi_1$ corresponding to $\lambda_1$. Since $\phi_1>0$ in $\Omega_n$ and $u_n(-n,\cdot)<u_n(n,\cdot)$ in $(0,1)$, we still obtain that $\partial_{x_1}u_n>0$ in $\Omega_n$.

{\bf Step 3. Picking a reference point.} Note that any horizontal translation of a solution to \eqref{slee} is still a solution. Because of this, $u_n$ could converge to a trivial limit. To avoid this, we need to fix a point, say $(0,1/2)$,  at which every approximating solution takes a common value strictly between $\phi(1/2)$ and $\varphi(1/2)$.

For every $n$, by the strict monotonicity of $u_n(\cdot,1/2)$, there exists a unique number $a_n\in(-n,n)$ such that
\begin{align*}
u_n(a_n,1/2)=\frac12(\phi(1/2)+\varphi(1/2)).    
\end{align*}

The key ingredient of this step is to prove the following lemma.
\begin{lemma}\label{lemma_ref}
It holds that $\lim_{n\rightarrow\infty}(n-a_n)=+\infty$ and $\lim_{n\rightarrow\infty}(n+a_n)=+\infty$.
\begin{proof}
We argue by contradiction and assume that (up to a subsequence) $\lim_{n\rightarrow\infty}(n-a_n)=l$ for some $l\in[0,\infty)$. By \eqref{uniform_bound} and the Arzel\`a-Ascoli theorem, $w_n(x_1,x_2)=u_n(x_1+n,x_2)$ converges locally in $C^{2,\alpha}$ to a limit $w$ which solves
\begin{eqnarray*}
\left\{\begin{aligned}
&-\Delta w=f(w), & x\in (-\infty,0)\times(0,1),\\
&w(x_1,0)=0, & x_1\in (-\infty,0),\\
&w(x_1,1)=c, & x_1\in (-\infty,0),\\
&w(0,x_2)=\varphi(x_2), & x_2\in[0,1].
\end{aligned}\right.
\end{eqnarray*}
Moreover, $w$ satisfies that $\partial_{x_1}w\ge0$ and $w(-l,1/2)=\frac12(\phi(1/2)+\varphi(1/2))<\varphi(1/2)$. By maximum principle, $\partial_{x_1}w<0$ in $(-\infty,0)\times(0,1)$. We can define
\[
w_-(x_2)=\lim_{x_1\rightarrow-\infty}w(x_1,x_2).
\]

Observe, for later use, that the above limit holds in $C^1$ sense. Next, let us recall the Hamiltonian identity (see \cite{G_JFA_2008})
\[
\int_0^1 \frac12(|\partial_{x_2}w|^2-|\partial_{x_1}w|^2)-F(w) \,dx_2=const \quad \forall x_1\in (-\infty,0].
\]
Taking $x_1=0$ and letting $x_1\rightarrow-\infty$, respectively, we get
\begin{align*}
I(w_-)=I(\varphi)-\frac12\int_0^1 |\partial_{x_1}w(0,x_2)|^2\,dx_2.
\end{align*}
Since $I(w_-)\ge I(\varphi)$, we infer that
\begin{align*}
\int_0^1 |\partial_{x_1}w(0,x_2)|^2\,dx_2=0,
\end{align*}
and consequently, $\partial_{x_1}w(0,x_2)\equiv0$. 

On the other hand, $\tilde{w}\vcentcolon=\varphi-w$ satisfies
\begin{eqnarray*}
\left\{\begin{aligned}
&-\Delta \tilde{w}=q(x)\tilde{w}, & {\rm in}\ (-\infty,0)\times(0,1),\\
&\tilde{w}>0, & {\rm in}\ (-\infty,0)\times(0,1),\\
&\tilde{w}=0, & {\rm on}\ \partial((-\infty,0)\times(0,1)),
\end{aligned}\right.
\end{eqnarray*}
where
\begin{eqnarray*}
q(x)=\left\{\begin{aligned}
&\frac{f(\varphi(x))-f(w(x))}{\varphi(x)-w(x)}, & {\rm if}\ \varphi(x)\neq w(x),\\
&0, & {\rm if}\ \varphi(x)= w(x)
\end{aligned}\right.
\end{eqnarray*}
is a bounded function. By Hopf lemma, we have $-\partial_{x_1}v(0,x_2)=\partial_{x_1}w(0,x_2)<0$ for $x_2\in (0,1)$. This leads to a contradiction.

A similar argument can prove the second part that $\lim_{n\rightarrow\infty}(n+a_n)=+\infty$.
\end{proof}
\end{lemma}

{\bf Step 4. Passing to the limit.}
Define $w_n(x_1,x_2)=u_n(x_1+a_n,x_2)$. Remember that 
\[
w_n(0,1/2)=\frac12(\phi(1/2)+\varphi(1/2)).
\]
By \eqref{approx_equ}, \eqref{uniform_bound}, Lemma \ref{lemma_ref}, and Arzel\`a-Ascoli theorem, we get that $w_n$ (up to a subsequence) converges in $C_{loc}^{2}(\overline{\Omega_\infty})$ to some limit $u$ that solves \eqref{slee}. Moreover, $u$ satisfies that $\partial_{x_1}u\ge0$, $\phi\le u\le\varphi$, and
\begin{align}\label{u_btw1}
\phi(1/2)<u(0,1/2)=\frac12(\phi(1/2)+\varphi(1/2))<\varphi(1/2).   
\end{align}

Let us define
\[
\underline{u}(x_2)=\lim_{x_1\rightarrow-\infty}u(x_1,x_2),\quad
\bar{u}(x_2)=\lim_{x_1\rightarrow+\infty}u(x_1,x_2).
\]
Clearly, $\underline{u}$ and $\bar{u}$ are two solutions to \eqref{ode}, and they satisfy
\begin{align}\label{u_btw2}
\phi\le \underline{u}\le u\le \bar{u}\le\varphi.
\end{align}

Next, for every $n$, we get again from the Hamiltonian identity that
\begin{align*}
I(\phi)=I(\varphi)
\ge\int_0^1 \frac12(|\partial_{x_2}w_n|^2-|\partial_{x_1}w_n|^2)-F(w_n) \,dx_2
\end{align*}
for all $x_1\in[-n-a_n,n-a_n]$.
Letting $n\rightarrow\infty$, then $x_1\rightarrow\pm\infty$, we arrive at
\begin{align*}
I(\phi)=I(\varphi)
\ge I(\underline{u})=I(\bar{u}).
\end{align*}
But $\phi$ and $\varphi$ are global minimizes of $I$, hence
\begin{align*}
I(\phi)=I(\varphi)=I(\underline{u})=I(\bar{u}).
\end{align*}
This together with \eqref{u_btw1}, \eqref{u_btw2} and Assumption \ref{two_ends} implies that $\underline{u}=\phi$ and $\bar{u}=\varphi$. Finally, the maximum principle implies that $\partial_{x_1}u>0$ in $\Omega_\infty$.

Hence the proof of Proposition \ref{pro_scheme} is completed.
\end{proof}


\section{Proof of Theorem \ref{main theorem}}\label{sec_pf1}

This section is devoted to the proof of Theorem \ref{main theorem}. For this, we apply Proposition \ref{pro_scheme} and the following proposition where we find a nonlinearity $f$ for which Assumption \ref{two_ends} is satisfied.

\begin{pro}\label{pro_findf}
There exist a nonlinearity $f\in C^\infty$ and a constant $c>0$ such that Assumption \ref{two_ends} is satisfied. Moreover, $\phi'>0$ in $[0,1]$, and there exists a number $t_1\in (0,1)$ such that $\varphi'>0$ in $[0,t_1)$ and $\varphi'<0$ in $(t_1,1]$.
\end{pro}

\begin{proof}

Our argument is inspired by the proof of   \cite[Proposition 3.1 and Lemma 3.2]{DR_arxiv_2024}, but new ingredients are necessary. The proof is divided into five steps.

{\bf Step 1.}
To start with, define $\chi:\R \to [0,+\infty)$ a $C^\infty$ function satisfying
\begin{equation}
	 \chi(s)=0\,\,\text{if}\,\, s \leq 1,
 \quad s-1 \geq \chi(s) \geq s-2\,\,\text{ for all } s \ge 1,
\end{equation}
and
\begin{equation}
\chi'(s) > 0\,\,\text{ for all }s >1.
\end{equation}

 Let
\[
F_{\lambda}(s)=\chi(s)^3-\lambda \, \chi(s)^4
\]
for any $\lambda \geq 0$ to be determined, and
\[
f_{\lambda}(s)= F_{\lambda}'(s)= \chi'(s)(3 \chi(s)^2 - 4 \lambda \chi(s)^3).
\]

We will work in the function space $\mcH_{0,1}(0,1)$:
\[
\mcH_{0,1}(0,1):=\{\psi\in H^{1}(0,1): \psi(0)=0, \psi(1)=1  \}.\]
Define a functional $I_\lambda: \mcH_{0,1}(0,1) \rightarrow \R$ as
\[
I_\lambda(\psi)=\int_{0}^{1}\left(\frac12|\psi'(t)|^2-F_{\lambda}(\psi(t))\right)\,dt.
\]

It is clear that critical points of $I_\lambda$ correspond to solutions of the Euler-Lagrange equation:

\begin{eqnarray}\label{ode-bis}
	\left\{\begin{aligned}
		&-\psi''=f_{\lambda}(\psi), & {\rm in}\ (0,1),\\
		&\psi(0)=0, \psi(1)=1.
	\end{aligned}\right.
\end{eqnarray} Observe that $\phi(t)=t$ is a solution of \eqref{ode-bis} for any $\lambda \geq 0$. Moreover, $\phi$ is a nondegenerate local minimum of $I_{\lambda}$, since for any $\xi \in H_0^1(0,1)$,
\begin{equation} \label{nondeg} 
I_{\lambda}''(\phi)[\xi, \xi]= \int_0^1 \xi'(t)^2 \, dt \geq C_p \| \xi\|_{H_0^1}^2,
\end{equation}
where $C_p>0$ is an absolute constant and we have used the Poincar\'e inequality. We have the following lemma. 

\begin{lemma} \label{lemaphi} If $\psi\neq \phi$ is another solution of \eqref{ode-bis} for some $\lambda>0$, then the following two properties hold.
	\begin{enumerate}
		\item[a)] $\psi(t) > \phi(t) $ for any $t \in (0,1)$.
	\item[b)] $\psi(t)$ is concave and achives its maximum, which is bigger than $1$.
	\end{enumerate}
\end{lemma}
\begin{proof}[Proof of Lemma \ref{lemaphi}]
In order to prove Part $a$), take $t_0 \in (0,1)$ so that 
		$$\psi(t_0) - \phi(t_0) = \min\{ \psi(t) - \phi(t): \ t \in [0,1]\} ,$$ and assume that this minimum is nonpositive. Observe that $\psi'(t_0)=1$ and $\psi(t_0) < 1$ so that $f_{\lambda}(\psi(t)) =0$ in a neighborhood of $t_0$. As a consequence $\psi(t)= t + \psi(t_0)- t_0$ in a neighborhood of $t_0$, but since all those values are below $1$, it turns out that $\psi(t)= t + \psi(t_0)- t_0$ for all $t \in [0,1]$. By the boundary conditions, we get $\psi(t)=t$, which contradicts $\psi\neq \phi$. 
	
	We now show Part $b$). First, if $\psi(t) \leq 1$ for all $t \in [0,1]$, then $f_{\lambda}(\psi(t))=0$ and hence $\psi$ is linear, so that $\psi=\phi$. Then, take $t_1 \in (0,1)$ so that $\psi(t_1)= \max \{ \psi(t): \ t \in [0,1]\} >1$.
	
	Observe now that there exists an $s(\lambda)>0$ such that  $f_\lambda(s) < 0$ for all $s >s(\lambda)$ and $f_\lambda(s) \geq 0$ for all $s \leq s(\lambda)$. Since $\psi(t_1)$ is a maximum, one has $f_\lambda(\psi(t_1))=-\psi''(t_1) \ge 0$. As a consequence $\psi(t_1) \leq s(\lambda)$ and hence $\psi$ is concave.
\end{proof}

 Define
\[
m_\lambda=\inf_{\mcH_{0,1}(0,1)}I_\lambda.
\]
One easily sees that $m_\lambda\le 1/2$ for all $\lambda\ge0$ since $I_\lambda(\phi)=1/2$.

{\bf Step 2.} 
For every $\lambda>0$, we have $-\infty<m_\lambda\le 1/2$, and $m_\lambda$ is achieved. Moreover, $m_\lambda\le m_{\lambda'}$ for $\lambda\le \lambda'$.

For all $\lambda>0$, since $s-1 \geq \chi(s) \geq s-2$ for all $s \ge1$, we have that $F_{\lambda}(t) \leq k_\lambda$ for some constant $k_\lambda \in \R$. Therefore, 
\begin{equation}
	\label{coercivity} I_{\lambda}(\psi) \geq 
	\int_{0}^{1}\left( \frac12|\psi'(t)|^2-k_{\lambda}\right) \, dt,
\end{equation}
and coercivity easily follows. Then the weak lower semicontinuity is a consequence of the fact that $F_\lambda\leq k_\lambda$ and Fatou's Lemma. Finally, the monotonicity of $m_\lambda$ with respect to $\lambda$ is immediate from its definition.

{\bf Step 3.} 
The set $\mathcal{E}=\{\lambda>0:m_\lambda< 1/2 \}$ is not empty and bounded from above, so we may denote by $\lambda^*$ its supremum.

We first show that $\mathcal{E}$ is not empty. Indeed, take any $w \in H_0^1(0,1)$, $v>0$, and define $\psi_\mu= \phi + \mu w \in \mcH_{0,1}(0,1)$. For sufficiently large $\mu >0$, we have that
\begin{align*}
I_0(\psi_{\mu})=& 
 \int_0^1 \left( \frac{1}{2} (1+ \mu w'(t))^2 - \chi(t + \mu w(t))^3 \right) \,dt\\ 
\leq& \int_0^1 \left[ \frac{1}{2} (1+ \mu w'(t))^2 - ( \mu w(t)-2)^3 \right]\,dt  <-1.
\end{align*}
Once $\mu$ is fixed, we can take $\lambda>0$ small enough so that $\lambda \int_{0}^1  \chi(\psi_{\mu}(t))^4 \, dt <1$ to conclude that $I_{\lambda}(\psi_{\mu})<0$. 

Let us now show that $I_{\lambda}(\psi) \geq 1/2$ for sufficiently large $\lambda$. Define $\mathcal{A}= \{t \in [0,1]: \ \psi(t) \leq 1\}$ and ${\mathcal{B}}= \{t \in [0,1]: \ \psi(t) > 1\}$. Then by Poincar\'e inequality, we obtain
$$
I_{\mathcal{A}}\vcentcolon=\int_{\mathcal{A}}\left(\frac12|\psi'(t)|^2-F_{\lambda}(\psi(t))\right)\,dt = \int_{\mathcal{A}} \frac12|\psi'(t)|^2 \geq \int_0^a \frac12|\psi'(t)|^2 \geq 1/2.
$$  
Above we denote $a$ such that $[0,a] \subset {\mathcal{A}}$ with $\psi(a)=1$. 

We now consider the integral in the set ${\mathcal{B}}$, assuming that ${\mathcal{B}} \neq \emptyset$. We apply Poincar\'{e} inequality to $\psi-1$ in ${\mathcal{B}}$: this is possible because $\psi=1$ on $\partial {\mathcal{B}}$.  Then:
\begin{align*}
I_{\mathcal{B}}\vcentcolon=&\int_{\mathcal{B}} \left(\frac12|\psi'(t)|^2-F_{\lambda}(\psi(t))\right)\,dt\\
\geq &\int_{\mathcal{B}}\left(\frac{C}{2}|\psi(t)-1|^2-F_{\lambda}(\psi(t))\right)\,dt\\
\geq &\int_{\mathcal{B}} \left( \frac{C}{2} \chi(\psi(t))^2- \chi(\psi(t))^3 + \lambda \chi(\psi(t))^4 \right)\,dt.
\end{align*}

Observe that $C$ above can be taken independently of the set $\mathcal{B}$ since $\mathcal{B} \subset [0,1]$. If $\lambda$ is sufficiently large, the function $ g(s):= \frac{C}{2} s^2- s^3 + \lambda  s^4$ is nonnegative, and hence $I_{\mathcal{B}}\geq 0$, which implies that
$$I_{\lambda}(\psi) = I_{\mathcal{A}}+I_{\mathcal{B}} \geq 1/2.$$

{\bf Step 4.} 
$m_{\lambda^*}=1/2$, and $I_{\lambda^*}$ has two global minimizers $\phi$ and $\varphi \in \mcH_{0,1}(0,1)$ (see Figure \ref{pic1}).

\begin{figure}[h]
	\centering 
	\begin{minipage}[c]{15cm}
		\centering
		\resizebox{12cm}{9cm}{\includegraphics{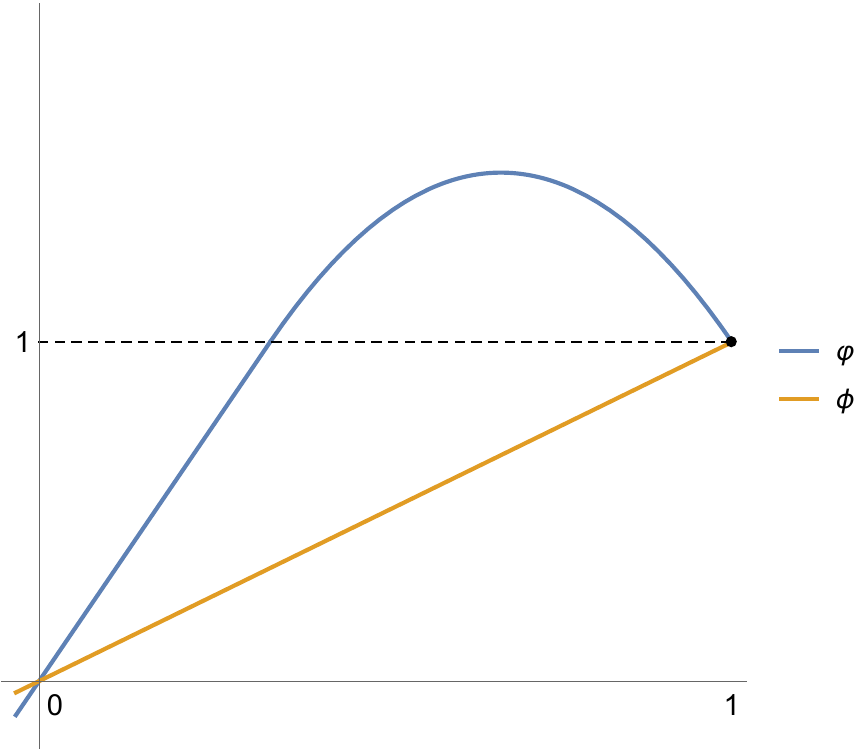}}
		\caption{The graphs of the functions $\phi$ and $\varphi$.}
        \label{pic1}
	\end{minipage}
\end{figure}

We start by pointing out that $m_{\lambda^*}=1/2$ by the definition of $\lambda^*$.
Consider a sequence $\{\lambda_n\}_{n\in\N}\subseteq\mathcal{E}$ such that $\lambda_n\to\lambda^*$ as $n\to+\infty$. Then there exists a sequence  $\{\varphi_{n}\}_{n\in\N}\subseteq \mcH_{0,1}(0,1)$ such that $I_{\lambda_n}(\varphi_{n})=m_{\lambda_n}<1/2$ for all $n\in\N$. 

Now, taking into account that $\lambda_n\to\lambda^*$, the inequality \eqref{coercivity} tells us that $\|\varphi_{\lambda_n}\|_{H^1}\le C$ for some constant $C>0$ and for all $n\in\N$. Then, up to subsequences, $\varphi_{n}\wto\varphi^*$ in $H^1(0,1)$ for some $\varphi^*\in \mcH_{0,1}(0,1)$. In particular, $\varphi_{n} \to\varphi^*$ uniformly, and hence $f_{\lambda_n}(\varphi_n) \to f_{\lambda^*}(\varphi^*)$ uniformly. By \eqref{ode-bis} we conclude that $\varphi_n \to \varphi^*$ in $H^1(0,1)$. This convergence, together with \eqref{nondeg}, implies that $\varphi^* \neq \phi$. Moreover $\varphi^*$ is a solution of \eqref{ode-bis} with $\lambda= \lambda^*$ and $I_{\lambda^*}(\varphi^*)=\frac12$.

\medskip 

{\bf Step 5.} For $\lambda= \lambda^*$, there exists $\varphi \in \mcH_{0,1}(0,1)$ a minimal minimizer above $\phi$: in other words:
\begin{enumerate}
	\item[a)] $\varphi > \phi$;
	\item[b)] $I_{\lambda^*}(\varphi)=1/2$;
	\item[c)] If $\psi \in \mcH_{0,1}(0,1)$ satisfies that $\psi >\phi$ and $I_{\lambda^*}(\psi)=1/2$, then $\psi \geq \varphi$.
\end{enumerate} 

The proof of this step is quite similar to \cite[Lemma 3.2]{DR_arxiv_2024}. Hence we just outline the key steps for completeness.

First, we claim that the set of global minima of $I_{\lambda^*}$ is totally ordered. To show this, assume $\psi_1,\psi_2\in \mcH_{0,1}(0,1)$ are two arbitrary minimizers of $I_{\lambda^*}$. Define $\xi_1=\min\{\psi_1,\psi_2 \}$ and $\xi_2=\max\{\psi_1,\psi_2 \}$. Note that
\[
I_{\lambda^*}(\xi_1)\ge\frac12, \
I_{\lambda^*}(\xi_2)\ge\frac12, \
I_{\lambda^*}(\xi_1)+I_{\lambda^*}(\xi_2)=
I_{\lambda^*}(\psi_1)+I_{\lambda^*}(\psi_2)=1.
\]
This implies $I_{\lambda^*}(\xi_1)=I_{\lambda^*}(\xi_2)=\frac12$, that is, $\xi_1$ and $\xi_2$ are also minimizers. Therefore, they both satisfy \eqref{ode-bis} with $\lambda=\lambda^*$ and are smooth in $(0,1)$. Clearly, $\xi_1$ coincides with $ \psi_j$ (for $j=1$ or $j=2$) in a set with nonempty interior: the uniqueness of solutions of the initial value problem implies that $\xi_1= \psi_j$ in the whole interval $(0,1)$. As a consequence, either $\xi_1\equiv \xi_2$ or  $\xi_1$ is strictly less than $\xi_2$ in the interval $(0, 1)$. Since $\psi_1$ and $\psi_2$ are chosen arbitrarily, the set of global minima of $I_{\lambda^*}$ is totally ordered.


Second, denote $\mathcal{M}$ the collection of all global minimizers of $I_{\lambda^*}$ above $\phi.$ Then there exists $\{\psi_n\}_{n\ge1}\subset\mathcal{M}$ with $\|\psi_n\|_{L^\infty}\rightarrow\inf_{\mathcal{M}}\|\psi\|_{L^\infty}$. The uniform boundedness of $\psi_n$ and the equation \eqref{ode-bis} imply that, up to a subsequence, $\psi_n\rightarrow\varphi$ in $C^{2,\alpha}$ sense. By \eqref{nondeg} and Lemma \ref{lemaphi}, we infer that $\varphi\in\mathcal{M}.$ By the argument in the previous step and the fact that $\|\varphi\|_{L^\infty}=\inf_{\mathcal{M}}\|\psi\|_{L^\infty}$, we finally conclude that $\varphi$ is the desired minimal solution above $\phi.$

 Finally, recalling Part b) of Lemma \ref{lemaphi}, we conclude the proof of Proposition \ref{pro_findf}.
\end{proof}

Theorem \ref{main theorem} easily follows from Propositions \ref{pro_scheme} and \ref{pro_findf}.

\begin{proof}[Proof of Theorem \ref{main theorem}]
By Proposition \ref{pro_scheme} and Proposition \ref{pro_findf}, there exists $f\in C^\infty$ such that \eqref{slee} (with $c=1$) has a bounded, strictly monotone and smooth solution $u$ such that
\begin{align*}
\lim_{x_1\rightarrow-\infty}u(x_1,x_2)=\phi(x_2),\quad
\lim_{x_1\rightarrow+\infty}u(x_1,x_2)=\varphi(x_2)
\end{align*}
uniformly in $x_2\in[0,1]$. 

Let $(\bv,P)$ be defined by \eqref{euler_solution}. Then $v_2>0$ in $\Omega_\infty.$ Since $\phi$ is increasing and $\varphi$ is concave and above $\phi,$ we observe that the level set $\{ u=1\}$ consists of the upper boundary of $\Omega_\infty$ and the graph of some function $x_2=h(x_1)$ defined on half line. Since any horizontal translation of $u$ is still a solution to \eqref{slee}, we can choose the domain of $x_2=h(x_1)$ to be $[0,\infty)$. Finally, by Hopf's lemma, we infer that $v_1(\cdot,0)=-\partial_{x_2}u(\cdot,0)<0,$ $v_1(\cdot,1)=-\partial_{x_2}u(\cdot,1)>0$ on $(0,\infty)$, and $v_1(\cdot,1)=-\partial_{x_2}u(\cdot,1)<0$ on $(-\infty,0)$.
\end{proof}


\section{Proof of Theorem \ref{theostable}}\label{sec_pf2}

This section is devoted to the proof of Theorem \ref{theostable}, which will be a consequence of Proposition \ref{pro_scheme} and the following proposition.
	
\begin{pro} 
\label{pro-stable} There exists a nonlinearity $f\in C^\infty$ such that Assumption \ref{two_ends} is satisfied for $c=0$. Moreover, both $\phi$ and $\varphi$ are strictly positive in $(0,1)$.
\end{pro}

The arguments here are very similar to those of the previous section, so we will be sketchy at some points.
	
	\begin{proof}[Proof of Proposition \ref{pro-stable}]
	The proof is divided into five steps.

{\bf Stpe 1.}	
The same notations of $\chi(t)$, $F_{\lambda}(t)$, and $f_{\lambda}(t)$ are used as in the previous section. 
	We now work in the function space $H_{0}^1(0,1)$, where we define our functional: \[
	\hat{I}_\lambda(\psi)=\int_{0}^{1}\left(\frac12|\psi'(t)|^2 - 2 \psi(t)- F_{\lambda}(\psi(t))\right)\,dt,
	\]
	
	It is clear that critical points of $I_\lambda$ correspond to solutions of the Euler-Lagrange equation:
	
	\begin{eqnarray}\label{ode-tris}
		\left\{\begin{aligned}
			&-\psi''=2+f_{\lambda}(\psi), & {\rm in}\ (0,1),\\
			&\psi(0)=0, \psi(1)=0.
		\end{aligned}\right.
	\end{eqnarray} Observe that $\phi(t)=t(1-t)$ is a solution of \eqref{ode-tris} for any $\lambda \geq 0$ and it corresponds to a nondegenerate local minimum of $\hat{I}_{\lambda}$: indeed, for any $\xi \in H_0^1(0,1)$,
	$$ \hat{I}_{\lambda}''(\phi)[\xi, \xi]= \int_0^1 \xi'(t)^2 \, dt \geq c \| \xi\|_{H_0^1}^2.$$
	
	We have the following lemma which is analogous to Lemma \ref{lemaphi}.
	
\begin{lemma} \label{lemmaphi2} 
If $\psi\neq \phi$ is another solution of \eqref{ode-tris} for some $\lambda>0$, then $\psi(t) > \phi(t) $ for any $t \in (0,1)$.
Moreover, $\psi$ is symmetric about $t=\frac12$ and $\psi'>0$ on $[0,\frac12).$ 
\end{lemma}
	
\begin{proof}	

The proof of the first part of this lemma is also analogous to that of Lemma \ref{lemaphi} and will be omitted. In particular, $\psi>0$ in $(0,1)$ and $\psi'(0)>1$. For the second part, let us denote by $t_0>0$ the smallest point so that  $\psi'(t_0)=0$. By definition, $\psi$ is increasing in $(0, t_0)$. Observe now that the functions:
$$ \psi(t), \ \tilde{\psi}(t)=\psi(2 t_0-t),$$ 
are solutions to the same ODE, and $\psi(t_0)= \tilde{\psi}(t_0)$, $\psi'(t_0)= \tilde{\psi}'(t_0)=0$. By uniqueness of the initial value problem, we conclude that $\psi= \tilde{\psi}$. This implies that $\psi$ is positive and decreasing in the interval $(t_0, 2t_0)$, and $\psi(2t_0)=0$, which implies that $t_0=1/2$.
\end{proof}

	 Define
	\[
	\hat{m}_\lambda=\inf_{\psi\in H_{0}(0,1)} \hat{I}_\lambda(\psi).
	\]
	One easily sees that $\hat{m}_\lambda\le -1/6$ for all $\lambda\ge0$ since $I_\lambda(\phi)= -1/6$.

{\bf Step 2.} 
For every $\lambda>0$, we have $-\infty<\hat{m}_\lambda\le -1/6$, and $\hat{m}_\lambda$ is achieved. Moreover, $\hat{m}_\lambda\le \hat{m}_{\lambda'}$ for $\lambda\le \lambda'$.

The proof can be easily adapted from the previous section.

\medskip 

{\bf Step 3.} 
The set $\mathcal{E}=\{\lambda>0: \hat{m}_\lambda< -1/6 \}$ is not empty and bounded from above, so we may denote by $\hat{\lambda}^*$ its supremum.

We evaluate the functional $\hat{I}_0$ on $\mu \phi$, for large $\mu>0$:
\begin{align*}
\hat{I}_0(\mu \phi) =& \int_{0}^{1}\left(\frac {\mu^2}{2}|\phi'(t)|^2 - 2 \mu \phi(t)- \chi(\mu \phi(t))^3\right)\,dt \\
\leq& \int_{0}^{1}\left(\frac {\mu^2}{2}|\phi'(t)|^2 - 2 \mu \phi(t)- (\mu \phi(t)-2)^3\right)\,dt.    
\end{align*}
Hence, for sufficiently large $\mu$ we conclude that $\hat{I}_0(\mu \phi) < -2$. Fixed such value $\mu$, we can take $\lambda$ small enough so that $\hat{I}_\lambda(\mu \phi) < -1$, and hence $\mathcal{E}$ is not empty.

\medskip 

We now show that for $\lambda$ sufficiently large, $m_{\lambda}= -1/6$ and is achieved only at $\phi$. We reason by contradiction, and assume that there exists another minimizer $\psi_{\lambda} \neq \phi$ so that $\hat{I}_{\lambda}(\psi_{\lambda}) \leq -1/6$. In particular, $\psi_{\lambda}$ is a solution of \eqref{ode-tris}. By Lemma \ref{lemmaphi2}, $\psi_{\lambda}$ attains a maximum $\gamma_{\lambda}>1$, which must be attained at $p=1/2$. 

We now claim that $\chi (\psi_{\lambda}) \to 0$ uniformly. Since $\chi$ is increasing, it suffices to show that $\chi(\gamma_{\lambda}) \to 0$. 

Recall that $\psi_{\lambda}$ has a maximum at $1/2$, so that:
\begin{equation}\label{psilambda2ndd}
0 \leq - \psi_{\lambda}''(1/2) = 2 + \chi'(\gamma_{\lambda}) \left( 3 \chi(\gamma_{\lambda})^2 - 4 \lambda \chi(\gamma_{\lambda})^3\right).
\end{equation}

If $\chi(\gamma_{\lambda}) \nrightarrow 0$ then $\gamma_{\lambda} \nrightarrow 1$ and the right hand side of \eqref{psilambda2ndd} would diverge negatively as $\lambda \to + \infty$, a contradiction that proves the claim.

Therefore, one has
\begin{equation}
	\begin{aligned}
 -1/6 \geq & \hat{I}_{\lambda}(\psi_{\lambda}) = \int_{0}^{1}\left(\frac12|\psi_{\lambda}'(t)|^2 - 2 \psi_{\lambda}(t)- \chi(\psi_{\lambda})^3 + \lambda \chi(\psi_{\lambda})^4 \right) \, dt \\
 \geq &  \int_{0}^{1}\left(\frac12|\psi_{\lambda}'(t)|^2 - 2 \psi_{\lambda}(t)- \chi(\psi_{\lambda})^3 \right) \, dt = 
\int_{0}^{1}\left(\frac12|\psi_{\lambda}'(t)|^2 - 2 \psi_{\lambda}(t) \right) \, dt+ o(1).
\end{aligned}
\end{equation}
 This implies that $\psi_{\lambda}$ is a minimizing sequence of the limit functional:
$$ \tilde{I}(\psi)= \int_{0}^{1}\left(\frac12|\psi'(t)|^2 - 2 \psi(t) \right) \, dt.$$
As a consequence $\psi_{\lambda} \to \phi$ in $H^1(0,1)$, but this is impossible since $\gamma_{\lambda} = \| \psi_{\lambda}\|_{L^{\infty}} >1.$

\medskip 

{\bf Step 4.} 
$\hat{m}_{\hat{\lambda}^*}=-1/6$, and $\hat{I}_{\hat{\lambda}^*}$ has two global minimizers $\phi$ and $\varphi \in H_{0}^1(0,1)$ (see Figure \ref{pic2}).

This proof can be done following Step 3 of the proof of Proposition \ref{pro_findf} in the previous section.

	\begin{figure}[h]
	\centering 
	\begin{minipage}[c]{15cm}
		\centering
		\resizebox{12cm}{9cm}{\includegraphics{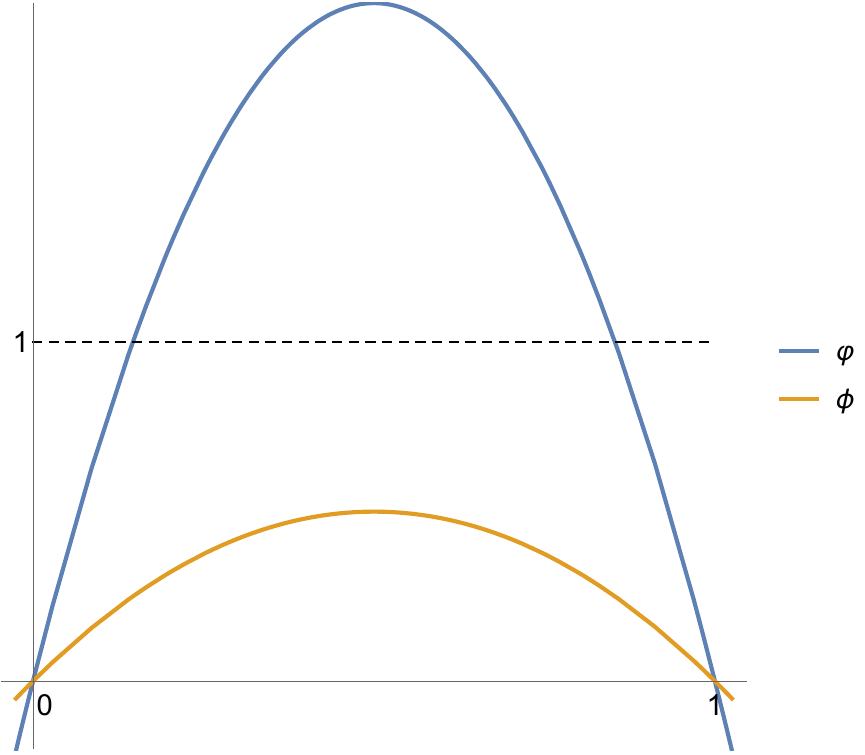}}
		\caption{The graphs of the functions $\phi$ and $\varphi$.}
        \label{pic2}
	\end{minipage}
\end{figure}

\medskip

{\bf Step 5.} For $\lambda= \hat{\lambda}^*$, there exists a $\varphi \in H_{0}^1(0,1)$ which is a minimal minimizer above $\phi$. In other words, the following properties hold.
\begin{enumerate}
	\item[a)] $\varphi > \phi$.
	\item[b)] $\hat{I}_{\hat{\lambda}^*}(\varphi)=-1/6$.
	\item[c)] If $\psi \in H_{0}^1(0,1)$ satisfies $\psi >\phi$ and $\hat{I}_{\hat{\lambda}^*}(\psi)=-1/6$, then $\psi \geq \varphi$.
\end{enumerate} 

The proof for  the existence of $\varphi$ follows the same reasoning as Step 4 in the previous section, and will be skipped.

With this the proof of Proposition \ref{pro-stable} is concluded.
	\end{proof}

	\begin{proof}[Proof of Theorem \ref{theostable}]
		
By applying Proposition \ref{pro_scheme} to the setting given in Proposition \ref{pro-stable}, we find a solution to the problem:

\begin{eqnarray}
	\left\{\begin{aligned}
		&-\Delta u=2+f_{\hat{\lambda}^*}(u), & \text{in  } \Omega_\infty,\\
		&u(x_1,0)=0, & x_1\in \mbR,\\
		&u(x_1,1)=0, & x_1\in \mbR.
	\end{aligned}\right.
\end{eqnarray}

Moreover, $\partial_{x_1}u >0$ in $\Omega_{\infty}$ and 

\begin{align*}
	\lim_{x_1\rightarrow-\infty}u(x_1,x_2)=\phi(x_2),\quad
	\lim_{x_1\rightarrow+\infty}u(x_1,x_2)=\varphi(x_2)
\end{align*}
uniformly in $x_2\in[0,1].$

Take now $\alpha \in (0, \| \phi \|_{L^{\infty}}]$. Observe that $u(x) > \| \phi \|_{L^{\infty}}$ on the line $L=\R \times \{1/2\}$ by monotonicity, and hence the superlevel set $U^\alpha=\{x \in \Omega_{\infty}: \ u(x)>\alpha\}$ contains $L$.

We now claim that if $U^{\alpha}$ is convex, then it is a strip in the form $U^{\alpha}= \R \times (a,b)$, with $a <1/2 < b$. Indeed, define
$$a = \inf \{ x_2 \in (0,1), \mbox{ such that } (x_1, x_2) \in U^{\alpha} \mbox{ for some } x_1 \in \R \},$$
$$ b = \sup \{ x_2 \in (0,1), \mbox{ such that } (x_1, x_2) \in U^{\alpha} \mbox{ for some } x_1 \in \R \}.$$

Clearly, $a<1/2<b$ and $U^{\alpha} \subset \R \times (a,b)$.  Take now $\e >0$ arbitrarily small, and fix $(x_1, x_2) \in U^{\alpha}$ such that $ x_2 \in (b- \e, b)$. By convexity, $U^{\alpha}$ contains all segments $[(x_1, x_2), (y_1, 1/2)]$ for all $y_1 \in \R$. As a consequence, 
$$ U^{\alpha} \supset \R \times [1/2, b-\e].$$

Since $\e>0$ is arbitrary, we conclude that $U^{\alpha} \supset \R \times [1/2, b)$. Reasoning analogously, we obtain $U^{\alpha} \supset \R \times (a, 1/2]$. In sum, we obtain $U^{\alpha} = \R \times (a,b)$. Since $\alpha>0$ we also get $0<a$, $b<1$.

Taking limits as $x_1 \to \pm \infty$ we conclude that $\phi(a)= \varphi(a)$ and $\phi(b) = \varphi(b)$, a contradiction that shows that $U^{\alpha}$ is not convex.
Hence the proof of Theorem \ref{theostable} is completed.
\end{proof}

\medskip
	
{\bf Acknowledgements.} The research of Gui is supported by University of Macau research grants  CPG2024-00016-FST, CPG2025-00032-FST, SRG2023-00011-FST, MYRG-GRG2023-00139-FST-UMDF, UMDF Professorial Fellowship of Mathematics, Macao SAR FDCT 0003/2023/RIA1 and Macao SAR FDCT 0024/2023/RIB1.  The research of Ruiz has been supported by
	the Grant PID2021-122122NB-I00 of the MICIN/AEI, the \emph{IMAG-Maria de Maeztu} Excellence Grant CEX2020-001105-M funded by MICIN/AEI, and the Research Group FQM-116 funded by J. Andaluc\'\i a. The research of Xie is partially supported by NSFC grants 12250710674 and 12426203, and Program of Shanghai Academic Research Leader 22XD1421400.

\bibliographystyle{plain}

\end{document}